\documentclass[12pt]{article}

\usepackage{amsmath}
\usepackage{amssymb}
\usepackage[mathscr]{eucal}
\usepackage{graphicx}
\usepackage{color} 

\usepackage{bm}
\usepackage{mathrsfs}
\usepackage{amsthm}
\usepackage{enumerate}
\usepackage[mathscr]{eucal}
\usepackage{eqlist}
\usepackage{graphicx}

\newtheorem{Theorem}{\sc Theorem}
\newtheorem{Definition}[Theorem]{\sc Definition}
\newtheorem{Proposition}[Theorem]{\sc Proposition}

\newtheorem{Corollary}[Theorem]{\sc Corollary}

\newtheorem{Problem}[Theorem]{\sc Problem}

\newcommand{\R}{{\if mm {\rm I}\mkern -3mu{\rm R}\else \leavevmode
		\hbox{I}\kern -.17em\hbox{R} \fi}}

\newcommand{\Div}{\mbox{\rm Div\,}}

\newcommand{\bu}{\mbox{\boldmath{$u$}}}
\newcommand{\bv}{\mbox{\boldmath{$v$}}}
\newcommand{\bw}{\mbox{\boldmath{$w$}}}
\newcommand{\bx}{\mbox{\boldmath{$x$}}}

\newcommand{\fb}{\mbox{\boldmath{$f$}}}

\newcommand{\bxi}{\mbox{\boldmath{$\xi$}}}

\newcommand{\bsigma}{\mbox{\boldmath{$\sigma$}}}
\newcommand{\btau}{\mbox{\boldmath{$\tau$}}}

\newcommand{\bnu}{\mbox{\boldmath{$\nu$}}}

\newcommand{\bzero}{\mbox{\boldmath{$0$}}}

\newcommand{\bz}{\mbox{\boldmath{$z$}}}

\def\sqr#1#2{{
		\vcenter{
			\vbox{\hrule height.#2pt
				\hbox{\vrule width.#2pt height#1pt \kern#1pt
					\vrule width.#2pt
				}
				\hrule height.#2pt
			}
		}
}}

\def\div{\mathop{\rm div}\nolimits}
\def\Div{\mathop{\rm Div}\nolimits}

\def\bar{\overline}
\def\real{\mathbb{R}}
\def\nat{\mathbb{N}}

\def\lista#1
{{ \itemindent 0.0cm \labelsep .2cm \leftmargin 0.8cm \rightmargin
		0.0cm \labelwidth 0.6cm \topsep 0.0mm
		\parsep 0.0mm
		\itemsep 0.0mm
		\begin{list}{}
			{ \setlength{\leftmargin}{.8cm} \setlength{\rightmargin}{0.0cm}
				\setlength{\parsep}{0.0mm} \setlength{\topsep}{.0mm}
				\setlength{\parskip}{.0cm} \setlength{\itemsep}{.0cm} }
			{#1}\end{list}} }

\begin{document}
	
\title{A generalized Stokes system with a nonsmooth slip boundary
condition\thanks{\, 
Project is supported by the European Union's Horizon 2020 Research and Innovation 
Programme under the Marie Sk{\l}odowska-Curie grant agreement No. 823731 CONMECH,
NSF of Guangxi Grant Nos. 2018GXNSFAA281353 and 2020GXNSFAA159052, 
the Beibu Gulf University Project No. 2018KYQD06, 
the Ministry of Science and Higher Education of Republic of Poland 
under Grant No. 440328/PnH2/2019, and 
the National Science Centre of Poland under Project No. 2021/41/B/ST1/01636.	
}}
	
\author{ 
Jing Zhao \footnote{\, 
College of Sciences, Beibu Gulf University, Qinzhou, Guangxi 535000, P.R. China.}, \ \ 
Stanis{\l}aw Mig\'orski \footnote{\,
College of Sciences, Beibu Gulf University, Qinzhou,
Guangxi 535000, P.R. China,	
and  Jagiellonian University in Krakow, Chair of Optimization and Control, ul. Lojasiewicza 6, 30348 Krakow, Poland.
E-mail address: stanislaw.migorski@uj.edu.pl.} \ \ \mbox{and} \ \ 
Sylwia Dudek \footnote{\, 
Department of Applied Mathematics,
Faculty of Computer Science and Telecommunications, 
Krakow University of Technology, 
ul. Warszawska 24, 31155 Krakow, Poland. 
E-mail address: sylwia.dudek@pk.edu.pl.}
}
	
\date{}
\maketitle
	
\noindent {\bf Abstract.} \
A class of quasi-variational-hemivariational inequalities in reflexive Banach spaces is studied.
The inequalities contain a convex potential, a locally Lipschitz superpotential, and an implicit
obstacle set of constraints. 
Results on the well posedness are established
including existence, uniqueness, dependence of solution on the data, and the compactness of the solution set in the strong topology. 
The applicability of the results is illustrated by the steady-state Stokes model of a generalized Newtonian incompressible fluid with a nonmonotone slip boundary condition.
	
\smallskip
	
	
\smallskip
	

	
\section{Introduction}\label{Section1}

In this paper we study the stationary Stokes equations with mixed boundary conditions which model a generalized Newtonian fluid of Bingham type. We deal with a nonmonotone version of the 
slip boundary condition described by the generalized subgradient of a locally Lipschitz 
potential. 
The paper is a continuation of our recent works~\cite{MD2021AMO,Zhao} in which the weak formulations lead to the  variational-hemivariational inequalities.
The novelty of the present paper is to consider the Stokes problem with an additional implicit obstacle constraint set depending on the solution.
This additional constraint makes the problem more involved since the resulting weak formulation  turns out to be a quasi variational-hemivariational inequality, see Problem~\ref{PV}. 
For the latter, we will prove an existence and compactness result in Theorem~\ref{Theorem1a}, 
and under stronger hypotheses,
a uniqueness result in Proposition~\ref{remkol2}.
Another novelty is a continuous dependence result, see~Theorem~\ref{Convergence} and Corollary~\ref{C123b}, 
in the strong topology, 
in contrast to~\cite{MD2021AMO,Zhao}, where the compactness of the solution set was established in the weak topology.

In the paper we use the basic material 
following~\cite{DMP1,DMP2,MOSbook,NP}.
Let $(X, \| \cdot \|_X)$ be a reflexive Banach space with its topological dual denoted by $X^*$. 
The notation $\langle\cdot,\cdot\rangle_{X^*\times X}$
stands for the duality brackets between $X^*$ and $X$.
Often, when no confusion arises, for simplicity, 
we omit the subscripts. 
Given a convex and lower semicontinuous function 
$\varphi \colon X \to \real$, the set defined by
\begin{equation*}
\partial \varphi (x) = 
\{ x^*\in X^* \mid \langle x^*, v -x \rangle
\le \varphi(v)-\varphi(x) \ \mbox{for all} \ v \in X\}
\end{equation*}
is called the (convex) subdifferential of $\varphi$ 
and an element $x^* \in \partial \varphi (x)$
is called a subgra\-dient of $\varphi$ at $x$.
Let $h\colon X \to \mathbb{R}$ be a locally Lipschitz function.
The generalized subgradient of $h$ at $x$ is given by
\[
\partial h (x) = \{\, \zeta \in X^* \mid h^{0}(x; v) \ge
{\langle \zeta, v \rangle} \ \mbox{for all} \ v \in X \, \},
\]
where 
\begin{equation*}
h^{0}(x; v) = \limsup_{y \to x, \ \lambda \downarrow 0}
\frac{h(y + \lambda v) - h(y)}{\lambda}
\end{equation*}
denotes the generalized (Clarke) directional derivative of $h$ at the point $x \in X$ in the direction $v \in X$. 
A locally Lipschitz function $h$ is said to be 
(Clarke) regular at the point $x \in X$ if for all $v \in X$ the 
derivative $h' (x; v)$ exists and 
$h^0(x; v) = h'(x; v)$.

A space $X$ with the weak topology is denoted by $X_w$.
The symbols $\rightharpoonup$ and $\to$ denote the weak convergence and the strong convergence, respectively.
If $U \subset X$, we write  
$\| U \|_X = \sup \{ \| x \|_X \mid x \in U \}$.
Given Banach spaces $X$ and $Y$, the notation 
${\mathcal L}(X, Y)$ stands for the set of all linear bounded operators from $X$ to $Y$.
For $A \in {\mathcal L}(X, Y)$, 
its adjoint operator $A^* \in {\mathcal L}(Y^*, X^*)$ 
is defined by
$\langle A^* y^*, x \rangle := \langle y^*, A x \rangle$ 
for every $y^* \in Y^*$ and $x \in X$.
The operator norm in ${\mathcal L(X, Y)}$ is denoted 
by $\| A \|$.


Let $X$ be a Banach space and 
$A \colon X \to X^*$ be an operator. Then

\noindent 
(i) 
$A$ is monotone,
if $\langle Au - A v, u-v \rangle_{X^* \times X} \ge 0$
for all $u$, $v \in X$,

\noindent 
(ii) 
$A$ is maximal monotone, if it is monotone, and
$\langle Au - w, u - v \rangle_{X^* \times X} \ge 0$ for any $u \in X$ implies $w = Av$,

\noindent 
(iii) 
$A$ is pseudomonotone,
if it is bounded (maps bounded sets into bounded sets) 
and $u_n \rightharpoonup u$ in $X$ with
$\displaystyle \limsup \langle A u_n, u_n -u \rangle_{X^* \times X} \le 0$
imply $\displaystyle \langle A u, u - v \rangle_{X^* \times X} \le
\liminf \langle A u_n, u_n - v \rangle_{X^* \times X}$ for all $v \in X$.
Equivalently, see, e.g., \cite[Proposition~3.66]{MOSbook},  
if $X$ is a reflexive Banach space, then 
$A$ is pseudomonotone, 
if and only if it is bounded and 
$u_n \rightharpoonup u$ in $X$ with 
$\displaystyle \limsup\,\langle A u_n, u_n - u \rangle_{X^* \times X} \le 0$ imply 
$\displaystyle \lim \,\langle A u_n, u_n - u \rangle_{X^* \times X} = 0$
and $A u_n \rightharpoonup A u$ in $X^*$.
%
%
%
\begin{Definition}\label{DefMosco} 
{\rm (see~\cite{DMP2,Mosco})}
Let $Y$ be a normed space. A sequence $\{ C_n \}$ of closed and convex sets in $Y$, is said to converge in the Mosco sense to a closed and convex set $C \subset Y$, denoted by
$C_n \ \stackrel{M} {\longrightarrow} \ C$
as $n \to \infty$, if 
	
	\smallskip
	
	\noindent	
	{\rm $(m_1)$}\ \
	for any $z_{n} \in C_{n}$
	with $z_{n} \rightharpoonup z$ in $Y$, up to a subsequence, we have $z \in C$,
	
	\smallskip
	
	\noindent 
	{\rm $(m_2)$}\ \
	for any $z \in C$, there exists $z_n \in C_n$ with $z_n \to z$ in $Y$.
\end{Definition}


Finally we recall the Kakutani--Ky Fan fixed point 
theorem in a reflexive Banach space, 
see, e.g.,~\cite[Corollary~1.7.42]{DMP2}. 
\begin{Theorem}\label{KKF}
	Given a reflexive Banach space $Y$ and a nonempty, bounded, closed and convex set $D \subseteq Y$. Let $\Lambda \colon D \to 2^D$
	be a set-valued map with nonempty, closed and convex values such that its graph is sequentially closed in $Y_w \times Y_w$ topology.
	Then $\Lambda$ has a fixed point.
\end{Theorem}

\section{The Stokes model for the Bingham type fluid}\label{Section2}

The Stokes equations form a system describing the flow of a fluid. They can be deduced from the nonlinear 
Navier-Stokes equations when the flow is slow and the fluid is incompressible. Then the convective term is small and can be neglected. 
We consider the weak formulation of the stationary Stokes problem with mixed boundary conditions which model a generalized Newtonian fluid of Bingham type. 
Due to the boundary conditions and additional constraints, 
the mathematical model leads naturally to an elliptic quasi variational-hemivariational inequality with constraints. 

We suppose that an incompressible fluid is
moving within a bounded domain (open and connected set) $\Omega$ in $\real^d$ with $d = 2$ and $d = 3$.
The boundary $\Gamma= \partial \Omega$ is supposed to be Lipschitz and partitioned into two disjoint, smooth and measurable parts $\Gamma_0$ and  
$\Gamma_1$ such that $|\Gamma_0| > 0$. 
The classical formulation of the Stokes flow problem reads as follows.
\begin{Problem}\label{P1} 
Find a flow velocity $\bu \colon \Omega \to \mathbb{R}^d$, 
an extra stress tensor $\mathbb{S} \colon \mathbb{M}^d \to \mathbb{M}^d$, and a pressure 
$p \colon \Omega \to \mathbb{R}$ such that 
$\bu \in K(\bu)$ and
	\begin{align}
	&- \Div {\mathbb{S}} + \nabla p = \fb
	&\mbox{\rm in} \ &\ \ \Omega, \label{2.1} \\[1mm]
	&\left\{\begin{array}{lll}
	\displaystyle 
	\mathbb{S} = {\mathbb{T}}(\mathbb{D}\bu)  
	+ g \, \frac{\mathbb{D}\bu}{\| \mathbb{D}\bu\|}
	\qquad 
	\ \, \ \ \ \mbox{\rm if} \ \ \ \mathbb{D}\bu \not= \bzero \\[3mm]
	\| \mathbb{S} \| \le g 
	\qquad\qquad \qquad \qquad \, \,  
	\ \ \ \mbox{\rm if} \ \ \ 
	\mathbb{D} \bu = \bzero
	\end{array}\right.
	&{\rm in} \ & \ \ \Omega, \label{2.2} \\[1mm]	
	&\div \bu = 0
	&\mbox{\rm in} \ & \ \ \Omega, \label{2.4a}\\[1mm]
	&\bu=0
	&\mbox{\rm on} \ & \ \ \Gamma_0, \label{2.6}\\[1mm]
	&\left\{\begin{array}{lll}
	u_\nu = 0 \\[1mm]
	-\btau_\tau(\bu) \in h(\bu_\tau) 
	\partial j_\tau (\bu_\tau)
	\end{array}\right.
	&{\rm on} \ & \ \ \Gamma_{1}, \label{2.7}
	\end{align}
\end{Problem}
\noindent 
where 
\begin{equation}\label{DEFU}
K(\bu)=\{ \, \bv \in V \mid 
k(\bv)\le r (\bu) \, \},  
\end{equation}
and the space $V$ is given in (\ref{definition:V}). 
%
Here 
$\bsigma (\bu,p) = \mathbb{S} (\mathbb{D} \bu) 
- p \, \mathbb{I}$ is the total stress tensor, $\mathbb{I}$ is the identity tensor, $p$ is the pressure, $\fb$ is called source term, and
$\mathbb{D} \bu = \frac{1}{2}
(\nabla \bu + \nabla \bu^\top)$ denotes the symmetric part of $\nabla \bu$ called also the deformation tensor.
A constitutive equation (\ref{2.2}) describes the Bingham type model in which $\mathbb{T}$ denotes a  constitutive function and $g \ge 0$ represents the plasticity threshold (yield stress).
It prescribes a maximal value $g$ (called the yield limit) that is the bound on the norm of the extra stress. If the strict inequality holds 
(the stress is low), 
there are no deformations and the fluid behave as a rigid body, when equality holds (at high stress), then the body initiates to behave as a fluid. 
The divergence-free condition (\ref{2.4a}) 
means that the fluid is incompressible, where 
the divergence of the velocity 
is given by $\div \bu = (u_{i,i})=0$.
The Dirichlet boundary condition
(\ref{2.6}) states that the fluid adheres 
to the wall $\Gamma_0$.
The first condition in (\ref{2.7}) is called the impermeability (no leak) boundary condition, 
while the second one is called the nonmonotone slip boundary condition. 
%
The traction vector on the boundary in (\ref{2.7}) 
is defined by
$$
\btau (\bu, p) = \bsigma(\bu, p) \bnu \ \ \mbox{on} \ \ \Gamma_1,
$$
where $\bnu$ stands for the unit outward normal vector 
on $\Gamma$.
The standard scalar products in $\real^d$ and $\mathbb{M}^d$ are denoted by $``\cdot"$ and 
$``:"$, respectively, where 
$\mathbb{M}^d$ is the class of symmetric 
$d \times d$ tensors. 
Normal and tangential components of the velocity vector are represented by  
$u_\nu = \bu \cdot \bnu$ and 
$\bu_{\tau} = \bu - u_\nu \bnu$, respectively, and 
therefore, we have the relations
\begin{equation*}
\tau_\nu (\bu, p) = \btau (\bu, p) \cdot \bnu 
\ \ \mbox{and} \ \  
\btau_\tau (\bu) = 
\btau (\bu, p) - \tau_\nu (\bu, p) \bnu
\ \ \mbox{on} \ \ \Gamma_1. 
\end{equation*}
Remark that $\btau$ and $\tau_\nu$ depend on $p$   
while $\btau_\tau$ is independent of the pressure. Therefore, we get
\begin{equation}\label{Relation}
\mathbb{S}_\nu (\bu)= \tau_\nu (\bu, p) + p
\ \ \mbox{and} \ \ 
\mathbb{S}_\tau (\bu) = \btau_\tau (\bu)
\ \ \mbox{on} \ \ \Gamma.
\end{equation}

We need the following hypotheses on the data.

\medskip

\noindent 
$\underline{H(T)}:$ \quad 
$\displaystyle {\mathbb{T}} \colon \Omega \times \mathbb{M}^d \to \mathbb{M}^d$ is a function such that

\smallskip

\lista{ 
	\item[(i)] 
	${\mathbb{T}}(\cdot, {\mathbb{D}})$ is measurable on $\Omega$ for all ${\mathbb{D}} \in \mathbb{M}^d$,
	\smallskip
	\item[(ii)] 
	${\mathbb{T}}(\bx, \cdot)$ is continuous on $\mathbb{M}^d$ for a.e. $\bx \in \Omega$,
	\smallskip
	\item[(iii)] 
	$\displaystyle
	\| {\mathbb{T}}(\bx, {\mathbb{D}}) \|_{\mathbb{M}^d} \le a_0(\bx) 
	+ a_1 \, \| {\mathbb{D}} \|_{\mathbb{M}^d}$
	for all ${\mathbb{D}} \in \mathbb{M}^d$, a.e. $\bx \in \Omega$ with $a_0 \in L^2(\Omega)$, $a_0$, $a_1 > 0$,
	\smallskip
	\item[(iv)] 
	$\displaystyle 
	({\mathbb{T}}(\bx,{\mathbb{C}}) 
	- {\mathbb{T}} (\bx, {\mathbb{D}}))
	: ({\mathbb{C}} - {\mathbb{D}}) 
	\ge m_T \, \|{\mathbb{C}} - {\mathbb{D}} \|_{\mathbb{M}^d}^2$ 
	for all ${\mathbb{C}}$, ${\mathbb{D}}
	\in \mathbb{M}^d$, 
	a.e. $\bx\in \Omega$ with $m_T > 0$. 
}

\medskip

\noindent 
$\underline{H(f, g)}:$ \quad
$\fb \in L^2(\Omega; \mathbb{R}^d)$, $g \in L^2(\Omega)$, $g \ge 0$.

\medskip 

\noindent
$\underline{H(h)}:$ \quad 
$h\colon \Gamma_1 \times \real^d \to \real$ 
is a function such that

\lista{
	\item[{\rm (i)}] 
	$h(\cdot, \bxi)$ is measurable on $\Gamma_1$ for all 
	$\bxi \in \real^d$, \smallskip
	\item[{\rm (ii)}] 
	$h(\bx, \cdot)$ is continuous on $\real^d$ for a.e. $\bx \in \Gamma_1$, \smallskip
	\item[{\rm (iii)}] 
	$0 < h_0 \le h(\bx, \bxi) \le h_1$ for all $\bxi \in \real^d$, a.e. $\bx \in \Gamma_1$.
}

\medskip

\noindent 
$\underline{H(j_\tau)}:$ \quad 
$j_\tau \colon \Gamma_1\times \real^d \to \real$ 
is a function such that 

\smallskip

\lista{
	\item[{\rm (i)}] 
	$j_\tau(\cdot,\bxi)$ is measurable on $\Gamma_1$ for all $\bxi\in \real^d$, \smallskip
	\item[{\rm (ii)}] 
	$j_\tau(\bx,\cdot)$ is locally Lipschitz for a.e. $\bx\in \Gamma_1$, \smallskip
	\item[{\rm (iii)}]  
	$\displaystyle 
	\| \partial j_\tau(\bx,\bxi) \|_{\real^d} \le b_0 (\bx) 
	+ b_1 \, \| \bxi \|_{\real^d}$ 
	\, for all \, $\bxi \in \real^d$, a.e. $\bx\in \Gamma_1$ with $b_0 \in L^2(\Gamma_1)$, $b_0$, $b_1 \ge 0$, \smallskip
	\item[{\rm (iv)}]
	either $j_\tau(\bx, \cdot)$ or $-j_\tau(\bx, \cdot)$ is regular for a.e. $\bx \in \Gamma_1$ 
	(see Section~\ref{Section1}).	
}

\medskip 

\noindent $\underline{H(k, r)}:$ \quad 
$k$, $r \colon V\to \real$ are functions such that

\smallskip

\lista{
	\item[{\rm (i)}] 
	$k$ is subadditive, positively homogeneous, 
	and weakly lower semicontinuous,
	\smallskip
	\item[{\rm (ii)}] 
	$r$ is weakly continuous, and
	$r(\bv) > 0$ for all $\bv \in V$.
}

\medskip

In hypothesis $H(j_\tau)$(iii), and in what follows, $\partial j$ denotes the ge\-ne\-ra\-lized gradient of the function $j$ with respect to its last variable.

When the constitutive function is of the form
\begin{equation}\label{const-law}
\mathbb{T} (\bx,\mathbb{D}) = 
\mu(\|\mathbb{D}\|) \mathbb{D} 
\ \ \ \mbox{for} \ \ \ \mathbb{D} \in {\mathbb{M}^d}, \ \mbox{a.e.} \ \bx \in \Omega 
\end{equation}
with $\mu \colon [0,\infty)\to \real$ a given viscosity function, then the model is called the generalized Newtonian fluid.
If $\mu(r) = \mu_0$ for $r \ge 0$ with $\mu_0 > 0$ 
a given viscosity constant, then (\ref{const-law}) 
reduces to $\mathbb{T}(\bx, \mathbb{D}) = \mu_0 \, \mathbb{D}$ which is the linear law for the usual Newtonian fluid which clearly satisfies $H(T)$. 
From the constitutive law (\ref{2.2}) one recovers 
the Bingham type model of Newtonian fluid 
(if $\mu (r) = \mu_0$ for $r \ge 0$)  
and the Navier-Stokes system (when $g = 0$). 
%
Under the hypothesis
	
\medskip
	
\noindent 
$\underline{H(\mu)}:$ \quad 
$\mu \colon [0, \infty) \to \real$ is such that
	
\smallskip
	
\lista{ 
\item[\rm (i)] 
$\mu$ is continuous and 
$0 < \mu_0 \le \mu(r) \le \mu_1$ for all $r \ge 0$, \smallskip 
\item[\rm (ii)] 
$(\mu(\| {\mathbb{C}}\|) {\mathbb{C}} 
- \mu(\| {\mathbb{D}} \|) {\mathbb{D}}) 
: ({\mathbb{C}} - {\mathbb{D}}) 
\ge \mu_2 \, \| {\mathbb{C}} - {\mathbb{D}} \|$ 
for all $\mathbb{C}$, $\mathbb{D} \in \mathbb{M}^d$ with $\mu_2 > 0$,
}
	
\medskip
	
\noindent 
the function ${\mathbb{T}}$ defined by $(\ref{const-law})$ satisfies
$H(T)$ with $a_0(\bx) = 0$ a.e., $a_1 = \mu_1$ and 
$m_T = \mu_2$.
Hypothesis $H(\mu)$ holds for typical models like the Carreau-type and power-law models,  see~\cite{BaoBar,BarLiu,Galdi,malek_book,malek}. 
Further, $H(\mu)$(ii) is satisfied when $\mu$ 
is monotonically increasing, see~\cite[Remark~3]{Baran2017}.
%
A particular version of Problem~\ref{P1} 
has been studied in~\cite{Zhao} 
for the Newtonian fluid when 
the yield limit $g=0$ and 
the linear constitutive function 
$\mathbb{T} (\bx,\mathbb{D}) = 
2 \, {\widetilde{\mu}} \, \mathbb{D}$ 
for $\mathbb{D} \in {\mathbb{M}^d}$, 
a.e. $\bx \in \Omega$,
where ${\widetilde{\mu}} > 0$ is a given viscosity, 
and without the additional constraint relation.


For the weak formulation we need the following spaces
\begin{equation}\label{definition:V}
V = \text{closure of}\ \widetilde{V}\ \text{in}\ H^1(\Omega;\real^d),  
\end{equation} 
\begin{equation*}
\widetilde{V}=\{ \, \bv\in C^\infty({\bar{\Omega}};\real^d) \mid 
\div \, \bv = 0 \ \text{in}\ \Omega,\  \bv = 0\ \text{on}\ \Gamma_0,\  v_\nu = 0 
\ \text{on}\ \Gamma_1 \, \}, \nonumber
\end{equation*}
\noindent 
and a set-valued map $K\colon V\to 2^V$ defined by
(\ref{DEFU}). 
%
From the Korn inequality, we know 
that two norms
$\| \bv \| = \| \bv \|_{H^{1}(\Omega;\real^d)}$ 
and
$\| \bv \|_V = \| \mathbb{D}\bv \|_{L^2(\Omega;\mathbb{M}^d)}$ 
are equivalent for $\bv \in V$.
It is well known that there exists the trace operator denoted by 
\begin{equation}\label{tracecompact}
\gamma\colon V \subset H^1(\Omega;\real^d) 
\to L^2(\Gamma;\mathbb{R}^d)
\end{equation} 
which is linear, continuous and compact,
see, e.g.,~\cite[Section~2.5.4, Theorems~5.5 and~5.7]{Necas}.
Its norm in the space 
$\mathcal{L}(V,L^2(\Gamma; \real^d))$
is denoted by $\|\gamma\|$. 
Moreover, instead of $\gamma \bv$, 
for simplicity, we often write $\bv$.


The weak formulation of Problem~\ref{P1} can be 
obtained by a procedure used in~\cite{Zhao}. 
We suppose that $\bu$, $\mathbb{S}$ and $p$ 
are sufficiently smooth functions that satisfy Problem~\ref{P1}. We multiply (\ref{2.1}) by 
$\bv - \bu$ with $\bv \in K(\bu)$, 
apply the Green formula, 
and use (\ref{2.2})--(\ref{2.7}) to get
\begin{Problem}\label{PV}
Find a velocity $\bu \in K(\bu)$ such that 
\begin{eqnarray*}
\displaystyle 
&&\hspace{-0.7cm}
\int_{\Omega} \mathbb{T} (\|{\mathbb{D}}\bu\|) : 
{\mathbb{D}} (\bv-\bu) \, dx 
+ \int_\Omega g \, (\| \mathbb{D}\bv\|
- \| \mathbb{D} \bu\|)\, dx \\
&&\hspace{-0.7cm}\qquad
+\int_{\Gamma_1} h(\bu_\tau ) j^0(\bu_\tau; \bv_\tau-\bu_\tau) \, d\Gamma \ge
\int_\Omega \fb \cdot (\bv-\bu) \, dx
\ \ \mbox{\rm for all} \ \ \bv \in K(\bu).
\end{eqnarray*}
\end{Problem}

The second boundary condition in (\ref{2.7}) describes a nonsmooth generalization of the Navier-Fujita slip condition.	
A prototype of (\ref{2.7}) is the linear slip condition 
of Navier of the form 
$-\btau_\tau(\bu) = k \, \bu_{\tau}$ on $\Gamma_1$ 
with $k > 0$ which was introduced in~\cite{Navier}. 
It simply states that the tangential velocity is proportional to the shear stress. 
Several variants of this law have been discussed 
in the literature, for instance:  
the nonlinear Navier-type slip condition,  
see~\cite{Leroux},
the Navier-Fujita condition of frictional type, 
see~\cite{Fu1,Fu2,Fu3,Saito,Safu},
the nonlinear Navier-Fujita slip condition,  
see~\cite{Leroux1}. 
In all of the aforementioned papers, the laws are 
modeled by condition (\ref{2.7}) with the convex potential $j \colon \real^d \to \real$, 
$j(\bxi) = \| \bxi \|$ for $\bxi \in \real^d$, 
where $\partial j$ stands for the convex 
subdifferential.

The condition (\ref{2.7}) is however more general and 
allows to deal with nonmonotone slip boundary
conditions of frictional type if $\partial j$ 
denotes the generalized subgradient of Clarke
for a nonconvex locally Lipschitz potentials.
For illustration, consider the following 
one dimensional example.  
Let $j_\lambda \colon \real \to \real$ be a potential 
in condition (\ref{2.7}) which depends on a positive parameter $\lambda > 0$ and defined by
$$
j_\lambda(r) =
\begin{cases}
\displaystyle 
\sqrt{|r|^2 + \lambda^2} - \lambda 
&\text{{\rm if} \ $|r| \le 1,$} \\[2mm]
\displaystyle
\bigg(\frac{1}{\sqrt{1+\lambda^2}} -1\bigg)|r|
+ \ln |r| + \sqrt{1+\lambda^2} - \lambda -
\frac{1}{\sqrt{1+\lambda^2}} +1
&\text{{\rm if} \ $|r| > 1$}
\end{cases}
$$
for $r \in \real$.
The derivative of $j_\lambda$ is given by
$$
j'_\lambda(r) =
\begin{cases}
\displaystyle
\frac{r}{\sqrt{|r|^2 + \lambda^2}}
&\text{{\rm if} \ $|r| \le 1,$} \\[3mm]
\displaystyle
\frac{1}{r} + 
\frac{1}{\sqrt{1+\lambda^2}} -1
&\text{{\rm if} \ $r > 1,$} \\[3mm]
\displaystyle
\frac{1}{r} -
\frac{1}{\sqrt{1+\lambda^2}} + 1
&\text{{\rm if} \ $r < -1$}
\end{cases}
$$
for $r \in \real$.
Note that $j'_\lambda$ is a continuous 
function, so $j_\lambda \in C^1(\real)$ and 
$|\partial j_\lambda (r)| = |j'_\lambda (r) | \le 1$ 
for $r\in \real$. 
Hence $j_\lambda$ is nonconvex and regular.  
The second condition in (\ref{2.7}) with the function $j_\lambda$ models the slip weakening phenomenon 
in which the tangential traction is a decreasing 
function of the tangential velocity.
It is clear that 
$j_\lambda$ satisfies $H(j_\tau)$ with 
$b_0(\bx) = 1$ and $b_1 = 0$.
On the interval $[-1, 1]$ the function 
$j_\lambda$ approximates, as $\lambda \to 0$, 
the convex and nondifferentiable function 
$r \mapsto |r|$. 
Therefore, $j'_\lambda$ on $[-1,1]$ 
approximates the monotone graph 
of
$\real \ni r \mapsto \partial |r| \in 2^\real$. 


The obstacle set (\ref{DEFU}) is introduced  
to take into account additional constraints 
on the solution. 
For instance, $k \colon V \to \real$ of the form  
$k(\bv) = \nu_0 \int_{\Omega} \| {\mathbb{D}} \bv \|^2 \, dx$ measures 
the rate dissipation energy due to viscosity, 
where $\nu_0> 0$ is the viscosity coefficient, 
while 
$k(\bv) = \int_{\Omega} \| \bv - \bv_0\|^2 \, dx$ 
represents the velocity tracking function. 
An example of the function 
$r \colon V \to \real$ is the following
$r(\bv) = \alpha + \int_{\Omega} \|\bv(x)\| \varrho(x)\, dx$
with $\varrho \in L^2(\Omega)$, $\varrho \ge 0$, 
and $\alpha > 0$.

\smallskip

For Problem~\ref{PV} we will provide results on the well posedness. To this end, in the next section, we introduce and analyze a quasi variational-hemivariational inequality 
in an abstract setting.

\section{Quasi variational-hemivariational inequality}\label{Section4}


Let $V$ be a reflexive Banach space with the dual
$V^*$. 
The norm in $V$ and the duality brackets for the pair 
$(V^*, V)$ are denoted by $\| \cdot \|$ and 
$\langle \cdot, \cdot \rangle$, respectively. 
Let $X$ be a Hilbert space with the norm $\| \cdot\|_{X}$ and the inner product 
$\langle \cdot, \cdot \rangle_X$.

Given 
a function $\phi \colon V \to \real$, 
operators
$A \colon V \to V^*$ and 
$M \colon V \to X$, 
a set-valued map
$U \colon V \to 2^V$ and $f \in V^*$,  
we consider the following problem.

\begin{Problem}\label{QVHI}
{\it Find $u \in V$ such that $u \in U(u)$ and}
\begin{equation*}
\langle A u - f, z - u \rangle
+ j^0(Mu, Mu; Mz - Mu) 
+ \, \phi (z) - \phi (u) \ge 0
\ \ \mbox{\rm for all} \ \, z \in U(u).
\end{equation*}
\end{Problem}
%
%

We need the following hypotheses on the data.

\medskip

\noindent 
$\underline{H(A)}:$ \quad 
$A \colon V \to V^*$ is an operator such that

\smallskip

\lista{ 
	\item[(i)] $A$ is pseudomonotone, \smallskip 
	\item[(ii)] 
	$\langle Av_1 - A v_2, v_1-v_2 \rangle \ge
	m_A \| v_1 - v_2\|^2$ for all $v_1$, $v_2 \in V$
	with $m_A > 0$, i.e., $A$ is strongly monotone.
}

\medskip

%

\noindent 
$\underline{H(j)}:$ \quad 
$j \colon X \times X \to \real$ is 
such that

\smallskip

\lista{ 
	\item[\rm (i)] 
	$j(w, \cdot)$ is locally Lipschitz for all 
	$w \in X$, \smallskip 
	\item[\rm (ii)] 
	$
	\| \partial j(w, v) \|_{X} \le
	d_1 + d_2 \| w\|_X + d_3 \| v \|_X$ 
	for all $w$, $v \in X$
	with $d_1$, $d_2$, $d_3 \ge 0$, \smallskip 
	\item[\rm (iii)]
	$X \times X \times X \ni (w, v, z) \mapsto 
	j^0(w, v; z) \in \real$ is upper semicontinuous.
} 

\medskip

\noindent 
$\underline{H(M)}:$ \quad
$M \colon V \to X$ is a linear, bounded, and compact operator.

\medskip

\noindent 
$\underline{H(U)}:$ \quad
$U \colon V \to 2^V$ is a set-valued map with nonempty, closed, convex values which is weakly Mosco continuous, i.e., 
for any $\{ v_n \} \subset V$ such that 
$v_n \rightharpoonup v$ in $V$, one has $U(v_n) \ \stackrel{M} {\longrightarrow} \ U(v)$, and  
$0 \in U(v)$ for all $v\in V$.

\medskip

\noindent 
$\underline{H(\phi)}:$ \quad 
$\phi \colon V \to \real$ is such that 

\smallskip

\lista{
	\item[(i)]
	$\phi$ is convex and lower semicontinuous on $V$, 
	\smallskip 
	\item[(ii)]
	there exists $c_\phi > 0$ such that 
	$
	\phi (v_1) - \phi(v_2) \le 
	c_\phi \, \| v_1 - v_2 \|$
	for all $v_1$, $v_2 \in V$.  
}


\noindent 
$\underline{H(f)}:$ \quad $f\in V^*$.

\medskip

We begin with the following existence result for Problem~$\ref{QVHI}$.
\begin{Theorem}\label{Theorem1}
Under the hypotheses
$H(A)$, $H(j)$, $H(M)$, $H(U)$,
$H(\phi)$, $H(f)$, and 
\begin{equation}\label{SSMMAALL} 
(d_2 + d_3) \, \| M \|^2 < m_A,
\end{equation}
Problem~$\ref{QVHI}$ has a solution. 
\end{Theorem}
\begin{proof}
It will be carried out in several steps.

{\bf Step~1}. 
We begin with an auxiliary elliptic quasi-variational inequality: 
find $u \in V$ such that $u \in U(u)$ and 
there is $w \in X$, $w \in \partial j(Mu,Mu)$, and
\begin{equation}\label{AUX1a}
\langle A u - f, z - u \rangle
+ \langle w, Mz - Mu \rangle_X 
+ \, \phi (z) - \phi (u) \ge 0
\ \ \mbox{\rm for all} \ \ z \in U(u).
\end{equation}
Observe that any solution to (\ref{AUX1a}) is also a solution to Problem~$\ref{QVHI}$. 
Indeed, let $u \in V$ be a solution to (\ref{AUX1a}). 
By the definition of the generalized gradient, we have
$\langle w, \xi \rangle_X \le j^0(Mu, Mu; \xi)$
for all $\xi \in X$ and
$$
\langle w, Mz - Mu \rangle_{X} \le
j^0(Mu, Mu; Mz - Mu) \ \ \mbox{for all} 
\ \, z \in U(u).
$$ 
Using the latter in the inequality (\ref{AUX1a}), 
we infer that $u \in V$ solves Problem~\ref{QVHI}.

{\bf Step~2}.
In view of Step~1, to finish the proof, it is enough 
to establish existence of solution to (\ref{AUX1a}).
To this end, let $(v, w) \in V \times X$ 
be fixed, and consider the following intermediate problem.
Find $u \in V$ such that $u \in U(v)$ and
\begin{equation}\label{INEQ1b}
\langle A u - f + M^* w, z - u \rangle 
+ \phi (z) - \phi (u) \ge 0 
\ \ \mbox{\rm for all} \ \, z \in U(v),
\end{equation}
where $M^* \colon X \to V^*$ denotes the operator adjoint to $M$.
We will prove that the inequality $(\ref{INEQ1b})$ 
has a unique solution $u \in V$ such that
\begin{equation}\label{INEQ22b}
\|u\| \le {\bar{c}}_1 + {\bar{c}}_3 \, \|w\|_X,
\end{equation}
where 
${\bar{c}}_1$, ${\bar{c}}_3$ 
are positive constants independent of $(v, w)$.
The inequality $(\ref{INEQ1b})$ is an elliptic variational inequality of the form: 
find an element $u \in K_1$ such that
\begin{equation}\label{VARIA}
\langle A u - {\widetilde{f}}, z - u \rangle 
+ \phi (z) - \phi(u) \ge 0 
\ \ \mbox{\rm for all} \ \ z \in K_1,
\end{equation}
where $K_1 = U(v)$ and 
${\widetilde{f}} = f - M^*w$.
By hypotheses, we know that $K_1$ is a nonempty, closed and convex subset of $V$ and ${\widetilde{f}} \in V^*$.
Moreover, by $H(A)$(ii), the operator $A$ is coercive 
in the following sense 
$$
\langle A v, v \rangle = 
\langle A v - A0,v \rangle 
+ \langle A 0,v \rangle \ge 
m_A \| v \|^2 + \|A 0\|_{V^*} \| v \|
\ \ \mbox{for all} \ \, v \in V.
$$
We apply~\cite[Theorem~4]{MOS30} to deduce that the problem (\ref{VARIA}), or equivalently (\ref{INEQ1b}),  
has a unique solution $u \in K_1 \subset V$.

We will show the estimate (\ref{INEQ22b}). 
We choose $0 \in U(v)$ as a test function 
in (\ref{VARIA}) to obtain 
\begin{equation}\label{equation34b}
\langle Au - A0, u \rangle 
\le 
\langle {\widetilde{f}} - A 0, u \rangle 
+ \phi(0) -\phi(u).
\end{equation} 
We take into account $H(A)$(ii) and $H(\phi)$(ii) 
to get 
\begin{equation*}
m_A \| u \|^2 \le 
\big( \| f - M^* w - A0\|_{V^*} + c_\phi \big) \| u \|.
\end{equation*}
and
\begin{equation}\label{equation36b}
m_A \| u \| \le \| f - A0 \|_{V^*} + c_\phi 
+ \| M \| \| w \|_X, 
\end{equation}
which implies (\ref{INEQ22b})
with
${\bar{c}}_1 : = m_A^{-1} \| f - A0 \|_{V^*} + c_\phi$ 
and 
${\bar{c}}_3 := m_A^{-1} \| M \|$. 
This completes the proof of Step~2.


{\bf Step~3}.
We consider a map $p \colon V \times X \to V$ 
defined by $p(v, w) = u$, where $u \in V$ is the unique solution to~(\ref{INEQ1b}) corresponding to 
$(v, w) \in V \times X$.
We shall show that 
$p$ is continuous 
from $V_w \times X_w$ to $V$.

Let $\{ v_n \} \subset V$, $\{ w_n \} \subset X$,
$v_n \rightharpoonup v$ in $V$, 
$w_n \rightharpoonup w$ in $X$, 
and $u_n = p(v_n, w_n) \in U(v_n)$. 
We prove that $u_n \to u$ in $V$ 
and $u = p(v, w) \in U(v)$.
Since $u_n \in V$ and $u_n \in U(v_n)$, 
we have 
\begin{equation}\label{equation37b}
\langle A u_n - {\widetilde{f}}_n, z - u_n \rangle 
+ \phi (z) - \phi (u_n) \ge 0 
\ \ \mbox{\rm for all} \ \ z \in U(v_n)
\end{equation}
with ${\widetilde{f}}_n := f - M^* w_n$. 
Now, taking advantage of the estimate proved in 
Step~2, we get the uniform estimate for the sequence 
of solutions $\{ u_n \}$ of the form 
\begin{equation}\label{equation38b}
m_A \| u_n \| \le \| f - A0 \|_{V^*} + c_\phi 
+ \| M \| \| w_n \|_X. 
\end{equation}
From (\ref{equation38b}), 
we see that $\{ u_n \}$ remains in a bounded set in $V$.
Moreover, since $V$ is reflexive, there exist some element
$u^* \in V$ and a subsequence of $\{ u_n \}$, 
still denoted in the same way, such that 
$u_n \rightharpoonup u^*$ in $V$. 
We can use condition $(m_1)$ in Definition~\ref{DefMosco}
of the Mosco convergence, 
and from $u_n \in U(v_n)$, $v_n \rightharpoonup v$ 
in $V$, and $H(U)$, we obtain $u^* \in U(v)$.

Next, let $z \in U(v)$.
We use condition $(m_2)$ in the Mosco convergence 
for $z \in U(v)$ and $u^* \in U(v)$ 
and we find two sequences 
$\{z_n\}$ and $\{\zeta_n\}$ with
\begin{equation}\label{Calpha1b}
z_n, \, \zeta_n \in U(v_n) \ \ \mbox{such that} 
\ \ z_n \to z 
\ \ \mbox{and} \ \ \zeta_n \to u^* \ \ \mbox{in} \ \ V, \ \mbox{as} \ n\to\infty .
\end{equation}
We choose 
$z = \zeta_n \in U(v_n)$ in (\ref{equation37b}) 
to obtain
\begin{equation}\label{CCC1b}
\langle A u_n, u_n - \zeta_n \rangle
\le \langle {\widetilde{f}}_n, u_n - \zeta_n \rangle
+\phi (\zeta_n) - \phi (u_n).
\end{equation}
Note that since $\phi$ is a continuous function, see~\cite[Theorem~5.2.8]{DMP1}, 
we have 
$\phi (\zeta_n) \to \phi (u^*)$.
From the weak lower semicontinuity of $\phi$ 
we get $\phi(u^*) \le \liminf \phi(u_n)$. 
These convergences entail
\begin{equation}\label{equation38b}
\limsup \big(
\phi (\zeta_n) - \phi(u_n) \big) 
= \lim \phi(\zeta_n) + \limsup \, (-\phi(u_n))
\le \phi(u^*) - \phi(u^*) =0.
\end{equation}
We use the convergence 	
${\widetilde{f}}_n := f - M^* w_n \to f - M^* w =: 
{\widetilde{f}}$ in $V^*$, and by (\ref{Calpha1b}), (\ref{CCC1b}) and (\ref{equation38b}), 
we get
\begin{eqnarray*}
&&\hspace{-1.2cm}
\limsup \langle A u_n , u_n - u^* \rangle 
\le 
\limsup \langle A u_n, u_n-\zeta_n \rangle
+ 
\limsup \langle Au_n, \zeta_n -u^* \rangle 
\\[2mm]
&&
\le 
\limsup
\Big(
\langle {\widetilde{f}}_n, u_n - \zeta_n \rangle +
\phi (\zeta_n) - \phi(u_n)\Big) 
+ 
\limsup \langle Au_n, \zeta_n -u^* \rangle \le 0. 
\end{eqnarray*}  
In conclusion, we have
$u_n \rightharpoonup u^*$ in $V$ and
$\limsup \langle A u_n , u_n - u^* \rangle \le 0$, 
which by the pseudomonotonicity of $A$ imply
\begin{equation}\label{equation39b}
\langle Au^*,u^*-v\rangle\le\liminf
\langle Au_n,u_n-v\rangle 
\ \ \mbox{for all} \ \ v \in V.	
\end{equation}

On the other hand, we take 
$z = z_n \in U(v_n)$ in (\ref{equation37b}) 
to obtain
\begin{equation}\label{equation40b}
\langle A u_n, u_n - z_n \rangle
\le - \langle {\widetilde{f}}_n, z_n - u_n \rangle
+\phi (z_n) - \phi (u_n).
\end{equation}
We use $H(\phi)$ and combine (\ref{equation39b}) with (\ref{equation40b}) 
to obtain
\begin{eqnarray*}
	&&\hspace{-0.5cm}
	\langle A u^* , u^* - z \rangle 
	\le \liminf \langle A u_n , u_n - z \rangle
	\le \limsup \langle A u_n , u_n - z \rangle \\[2mm]
	&&
	\hspace{-0.6cm}\quad 
	=\limsup 
	\langle A u_n , u_n - z \rangle
	+\lim \langle A u_n , z - z_n \rangle
	=\limsup \langle A u_n , u_n - z_n \rangle\\[2mm]
	&&\hspace{-0.6cm}
	\quad 
	\le - \lim \, \langle {\widetilde{f}}_n, z_n - u_n \rangle
	+ \limsup
	\big( \phi (z_n) - \phi (u_n)
	\big)
	\le - \langle {\widetilde{f}}, z-u^* \rangle +\phi(z)-\phi(u^*).
\end{eqnarray*}
In consequence, we have
$\langle A u^* - {\widetilde{f}}, u^* - z \rangle
\le \phi (z) - \phi (u^*)$ 
for all $z \in U(v)$. 
So we deduce that $u^* \in U(v)$ is a solution to the limit 
problem corresponding to (\ref{equation37b}), 
i.e., $u^* = p(v, w)$. 
The uniqueness of limit element 
$u^*$ implies that the whole sequence 
$\{ u_n \}$ convergences weakly to $u^*$ in $V$.

Subsequently, we show the strong convergence of
$\{ u_n \}$ to $u^*$ in $V$.
We can find a sequence 
$\{ \zeta_n \} \subset U(v_n)$ such that 
$\zeta_n \to u^*$ in $V$, as $n \to \infty$ (from condition $(m_2)$ of the Mosco convergence 
for $u^* \in U(v)$).
We choose $\zeta_n$ as a test function in
(\ref{equation37b}) to obtain 
$$
\langle A u_n - {\widetilde{f}}_n, 
\zeta_n - u_n \rangle
+\phi (\zeta_n) - \phi (u_n) \ge 0
\ \ \mbox{for all} \ \ n \in \nat,
$$
which implies
$$
\langle A u_n, u_n - \zeta_n \rangle
\le 
\langle {\widetilde{f}}_n, u_n - \zeta_n \rangle
+\phi (\zeta_n) - \phi (u_n).
$$
Using this inequality and $H(\phi)$, we have
\begin{eqnarray}\label{equation42b}
&&\hspace{-0.2cm}
\limsup \langle Au_n, u_n- u^* \rangle
\le 
\limsup \langle Au_n, u_n- \zeta_n \rangle 
\\[2mm]
&&
+ \limsup \langle Au_n, \zeta_n- u^* \rangle
\le 
\limsup \langle {\widetilde{f}}_n, 
u_n - \zeta_n \rangle 
\nonumber 
\\[2mm]
&&\quad
+ \limsup \, (\phi (\zeta_n) - \phi (u_n))
+ \limsup \langle Au_n, \zeta_n- u^* \rangle
\le 0. \nonumber 
\end{eqnarray}
Here, we have used the convergences 
$u_n-\zeta_n \rightharpoonup 0$ in $V$,
$\zeta_n\to u^*$ in $V$ and 
${\widetilde{f}}_n \to {\widetilde{f}}$ in $V^*$.
From $H(A)$(ii) and (\ref{equation42b}), it follows
\begin{eqnarray*}
	&&
	m_A \, \limsup \| u_n - u^* \|^2 \le
	\limsup \langle Au_n - Au^*, u_n - u^* \rangle
	\\[2mm]
	&&\quad
	\le
	\limsup \langle Au_n, u_n - u^* \rangle
	+
	\limsup \langle Au^*, u_n - u^* \rangle 
	\le 0.
\end{eqnarray*}
Hence $u_n \to u^*$ in $V$ as $n \to \infty$ 
which implies that the map $p$ is completely continuous.

\medskip

{\bf Step~4}. We shall use the fixed point argument.
We define the set-valued map $F \colon X \to 2^X$
by $F(z) := \partial j(z, z)$ for $z \in X$. 
The notation $\partial j(w, v)$ means the generalized gradient of $j(w, \cdot)$ at the point $v \in X$ for fixed $w \in X$.
We observe that for all $w$, $v \in X$ 
the set $\partial j(w, v)$ is nonempty, weakly compact, and convex in $X$, see, e.g.~\cite[Proposition~3.23(iv)]{MOSbook}. Hence, 
the values of $F$ are nonempty closed and convex in $X$.

We claim that the graph of $F$ is sequentially 
closed in $X \times X_w$. In fact, let 
$z_n \in X$, $z_n \to z$ in $X$, 
$z_n^* \in X$, $z^*_n \in F(z_n)$, and 
$z^*_n \rightharpoonup z^*$ in $X$. 
By the definition of the generalized gradient, 
we have
$\langle z^*_n, \xi \rangle \le j^0(z_n, z_n; \xi)$ 
for all $\xi \in X$.
Exploiting $H(j)$(iii), we get 
$$
\limsup \, 
\langle z^*_n, \xi \rangle \le \limsup j^0(z_n, z_n; \xi) 
\le j^0(z,z; \xi) \ \ \mbox{for all} \ \, \xi \in X,
$$
which implies $z^* \in \partial j(z, z) = F(z)$. 
Hence we get the desired closedness of the graph 
of $F$.
Moreover, it follows from $H(j)$(ii) that
$$
\| F(z) \|_X \le \| \partial j(z, z)\|_X 
\le d_1 + (d_2+d_3) \| z \|_X 
\ \ \mbox{for all} \ \, z \in X.
$$

Next, let	
\begin{equation}\label{equation44} 
D = \{ \, (v, w) \in V \times X \mid \|v \|\le r_1, \
\|w\|_X \le r_2 \, \} 
\end{equation}
with some $r_1$, $r_2 > 0$. 
We consider the set-valued map	
$\Lambda \colon D \to 2^D$ defined by
\begin{equation}\label{equation43}
\Lambda (v, w) :=
\big(
p(v, w), F(M p(v, w))
\big) = (u, F(Mu))
\ \ \mbox{for} \ \ (v, w) \in D.
\end{equation}
We establish some properties of map  $\Lambda$. 
First, we show that for suitable constants 
$r_1$, $r_2 > 0$, 
the values of the map $\Lambda$ lie in $D$.
So, we put 
\begin{equation*}
r_1 := 
\frac{C_1+ d_1 \| M\|}{m_A - (d_2+d_3) \| M\|^2}
\ \ \mbox{and} \ \ 
r_2 := d_1 + (d_2+d_3) \, \| M \| \, r_1
\end{equation*}
with  
$C_1 := \| A0 \|_{V^*} + \| f \|_{V^*} + 
c_\phi > 0$. 
Let
$\| v \| \le r_1$ and $\| w \|_X \le r_2$. 
Then, from (\ref{equation36b}), we have
\begin{eqnarray*}
	&&
	\hspace{-2.0cm}
	m_A \| u \| \le 
	C_1 + \| M \|\|w \|_X \le C_1 + \| M \| r_2  
	\\[2mm]
	&&
	\hspace{-1.5cm}
	\le
	C_1 + d_1 \| M \| + 
	\frac{(d_2+d_3)\| M \|^2 (C_1+d_1\|M\|)}
	{m_A - (d_2+d_3) \| M\|^2} = m_A \, r_1
\end{eqnarray*}
which implies $\| u \| \le r_1$. 
Further, we have
$$
\| F(Mu) \|_X \le d_1 + (d_2+d_3) \| M \|\| u \| 
\le
d_1 + (d_2+d_3) \| M \| r_1 = r_2.
$$
Hence, we have found positive constants $r_1$ and $r_2$ 
in the definition (\ref{equation44}) of the set $D$
such that 
$\Lambda (v, w) \subset D$ for all $(v, w) \in D$.
Moreover, the values of $\Lambda$ are nonempty, closed and convex sets, 
by the analogous properties of $F$.

Next, we prove that the graph of $\Lambda$
is sequentially weakly closed in $D \times D$.
Consider 
$(v_n, w_n) \in D$ such that
$(v_n, w_n) \rightharpoonup (v, w)$ in $V \times X$,
$({\bar{v}}_n, {\bar{w}}_n) \in \Lambda (v_n, w_n)$,
and
$({\bar{v}}_n, {\bar{w}}_n) 
\rightharpoonup ({\bar{v}}, {\bar{w}})$ in $V \times X$.
We show that
$({\bar{v}}, {\bar{w}}) \in \Lambda(v, w)$.
We have
\begin{equation}\label{CSSS}
{\bar{v}}_n = p(v_n, w_n) \ \ \ \mbox{and} \ \ \
{\bar{w}}_n \in F(M p(v_n, w_n)),
\end{equation}
by the definition of $\Lambda$. 
Using the continuity of the map $p$ 
and the continuity of the operator $M$, 
we obtain 
$
p(v_n, w_n) \to p(v, w)$ in $V$
and
$M p(v_n, w_n) \to M p(v, w)$ in $X$
which, together with (\ref{CSSS}) and  the closedness of the graph of $F$ in $X \times X_w$ topology, implies
\begin{equation*}
{\bar{v}} = p(v, w) \ \ \ \mbox{and} \ \ \
{\bar{w}} \in F(M p(v, w)).
\end{equation*}
Hence
$
({\bar{v}}, {\bar{w}}) \in
\big(
p(v, w), F(M p(v, w))
\big) = \Lambda (v, w)$, which proves the closedness of the graph of $\Lambda$.

Now we are in a position to apply 
the Kakutani--Ky Fan theorem with $Y = V \times X$ and the map $\Lambda$ given by (\ref{equation43}). 
In consequence, we deduce that there exists $(v^*, w^*) \in D$ 
such that $(v^*, w^*) \in \Lambda (v^*, w^*)$.
This means that $v^* = u_0$ and
$w^* \in F(M u_0)$, where
$u_0 \in V$, $u_0 \in U(u_0)$ and it satisfies
$$
\langle A u_0- f, z - u_0 \rangle
+ \phi (z) - \phi (u_0)
+ \langle M^* w^*, z - u_0 \rangle
\ge 0 
$$
for all $z \in U(u_0)$ with $w^* \in F(M u_0)$.
Hence, we conclude that $u_0 \in V$ 
is the solution to Problem~\ref{AUX1a}.
This completes the proof of the theorem.
\end{proof}

Now we investigate the dependence of the solution set to Problem~$\ref{QVHI}$ on the functions $(f, \phi)$. 
%
We consider a sequence of quasi 
variational-hemivariational inequalities: 
find $u \in V$ such that $u \in U(u)$ and
\begin{equation*}
\left\{
\begin{array}{lll}
\langle A u - f_n, z - u \rangle
+ j^0(Mu, Mu; Mz - Mu) 
\qquad\qquad \\[2mm] 
\qquad 
+ \, \phi_n (z) - \phi_n (u) \ge 0
\ \ \mbox{\rm for all} \ \ z \in U(u). 
\end{array}
\right.
\leqno{P(f_n, \phi_n)}
\end{equation*}

We need the following assumptions.

\medskip

\noindent 
$\underline{H(\phi)_1}:$ \quad 
$\phi$, $\phi_n \colon V \to \real$ is such that 

\smallskip

\lista{
	\item[(i)]
	$\phi$, $\phi_n \ \mbox{are convex and lower semicontinuous}$,  
	\smallskip 
	\item[(ii)]
	$\phi(v) \ge 0$, $\phi_n(v) \ge 0$ 
	for all $v \in V$ and $\phi(0) = \phi_n(0) = 0$, 
	\smallskip 
	\item[(iii)]
	there exists $c > 0$ such that for all $n \in \nat$,
	it holds 
	$$
	\phi (v_1) - \phi(v_2) \le c \, \| v_1 - v_2 \|, 
	\ \ 
	\phi_n(v_1) - \phi_n(v_2) \le c \, \| v_1 - v_2 \|
	\ \ \mbox{\rm for all} \ \ v_1, v_2 \in V, 
	$$
	\item[(iv)]
	$\limsup \, (\phi_n(v) - \phi_n(u_n)) \le \phi(v) - \phi(u)$ for all $u_n \rightharpoonup u$ in $V$, 
	and all $v \in V$.
}

\medskip

\noindent 
$\underline{H(f)_1}:$ \quad 
$f$, $f_n \in V^*$, $f_n \to f$ in $V^*$.

%
%

\begin{Theorem}\label{Convergence}
Let hypotheses
$H(A)$, $H(j)$, $H(M)$, $H(U)$, $H(\phi)_1$, $H(f)_1$,  
and $(\ref{SSMMAALL})$ hold, and 
%
%
let $\{ u_n \} \subset V$ be a sequence of solutions 
to problem $P(f_n, \phi_n)$.
Then, there is a subsequence of $\{ u_n \}$ 
that converges in $V$ to an element $u \in V$, 
where $u$ is a solution to problem $P(f, \phi)$.
\end{Theorem}
\begin{proof} 
Let $\{ u_n \} \subset V$ be a sequence of solution to problem $P(f_n, \phi_n)$. 
We claim that $\{ u_n \}$ belongs to a bounded set 
in $V$ independently of $n\in \nat$.
In fact, we take $0 \in U(u_n)$ as a test element 
in $P(f_n, \phi_n)$ and get
\begin{equation*}
\langle A u_n - A0, u_n \rangle
\le j^0(Mu_n, Mu_n; - Mu_n) + 
\phi_n (0) - \phi_n (u_n) + 
\langle f_n -A0, u_n \rangle,
\end{equation*}
which immediately, by $H(\phi)_1$(ii) and $H(A)$(ii),  implies 
$$
m_A \, \| u_n \|^2 \le 
j^0(Mu_n, Mu_n; - Mu_n) +
\| f_n - A0 \|_{V^*} \, \| u_n \|.
$$
Using $H(j)$(iii) 
and~\cite[Proposition~3.23(iii)]{MOSbook}, we obtain 
\begin{eqnarray*}
&&
j^0(Mu_n, Mu_n; - Mu_n)
\le \| \partial j_n (Mu_n, Mu_n)\| \| M \| \| u_n \| 
\\ [2mm]
&&\qquad 
\le d_1 \| M \| \| u_n \| +
(d_2 + d_3) \| M \|^2 \|u_n \|^2
\end{eqnarray*}
which, by (\ref{SSMMAALL}), entails
\begin{equation*}
\| u_n \| \le 
\frac{\| f_n - A0\|_{V^*} + d_1 \| M \|}
{m_A - (d_2 + d_3) \| M \|^2} =: L_n.
\end{equation*}
Recalling hypothesis $H(f)_1$, we know that $L_n$ is uniformly bounded by a constant independent of $n$, 
thus $\{ u_n \}$ is uniformly bounded in $V$ as claimed.  
We may assume that there exists 
$u^* \in V$ such that, at least for a subsequence, 
we have
\begin{equation}\label{weak44}
u_n \rightharpoonup u^* \ \ \mbox{in} \ \ V.
\end{equation} 

Next, we apply an argument from Step~2 of Theorem~\ref{Theorem1} and pass to the limit in 
$P(f_n, \phi_n)$. From (\ref{weak44}) 
and $H(U)$, we easily deduce that $u^* \in U(u^*)$.
Let $z \in U(u^*)$. Using the Mosco condition  $(m_2)$ 
in $H(U)$, we are able to find 
$\{\eta_n\}$, $\{z_n\} \subset U(u_n)$ such that
\begin{equation}\label{Approx1}
\eta_n \to u^* 
\ \ \mbox{and} \ \ z_n \to z \ \ \mbox{in} \ \ V, \ \mbox{as} \ n\to\infty .
\end{equation}
We select $\eta_n$ as a test element 
in $P(f_n, \phi_n)$ to obtain
\begin{equation}\label{AAA1b}
\langle A u_n, u_n -\eta_n \rangle 
\le 
\langle - f_n, \eta_n - u_n \rangle
+ j^0 (M u_n, M u_n; M \eta_n - M u_n) 
+\phi_n (\eta_n) - \phi_n (u_n).
\end{equation}
From hypothesis $H(\phi)_1$(iv) and (\ref{Approx1}), 
we get
\begin{equation}\label{AAA3b}
\limsup \, (\phi_n (\eta_n) - \phi_n(u_n)) 
\le \lim (\phi_n (\eta_n) - \phi_n(u^*)) 
+ \limsup (\phi_n (u^*) - \phi_n(u_n)) \le 0.
\end{equation}
Next, we use $H(f)_1$, $H(j)$(iii), (\ref{Approx1}), 
(\ref{AAA1b}), (\ref{AAA3b}) and 
the convergence 
$M u_n \to M u^*$ in $X$ 
to deduce 
\begin{eqnarray*}
&&\hspace{-1.2cm}
\limsup \langle A u_n , u_n - u^* \rangle 
=
\limsup \langle A u_n, u_n-\eta_n \rangle
+ 
\limsup \langle Au_n, \eta_n -u^* \rangle 
\\[2mm]
&&
\hspace{-0.6cm}
\le 
\lim \langle f_n, u_n-\eta_n \rangle
+ \limsup  j^0 (M u_n, M u_n; M \eta_n - M u_n)
\\[2mm]
&&
+ \limsup \,  (\phi_n (\eta_n) - \phi_n (u_n))
+ \limsup \langle Au_n, \eta_n -u^* \rangle \le 0. 
\end{eqnarray*}  
We are now in a position to use 
$u_n \rightharpoonup u^*$ in $V$, 
$\limsup \langle A u_n , u_n - u^* \rangle \le 0$, 
and the pseudomonotonicity of $A$ to have
\begin{equation}\label{AAA4b}
\langle Au^*,u^*-v\rangle\le\liminf 
\langle Au_n,u_n-v\rangle
\ \ \mbox{for all} \ \ v\in V.	
\end{equation}
Next, we choose 
$z_n \in U(u_n)$ as a test function 
in $P(f_n, \phi_n)$ to obtain
\begin{equation}\label{AAA5b}
\langle A u_n, u_n - z_n \rangle 
\le 
\langle - f_n, z_n - u_n \rangle
+ j^0 (M u_n, M u_n; M z_n - M u_n) 
+\phi_n (z_n) - \phi_n (u_n).
\end{equation}
Combining (\ref{Approx1}) and (\ref{AAA3b})--(\ref{AAA5b}), 
it follows  
\begin{eqnarray*}
&&\hspace{-1.0cm}
\langle A u^* , u^* - z \rangle 
\le \liminf \langle A u_n , u_n - z \rangle
\le \limsup \langle A u_n , u_n - z \rangle \\[2mm]
&&
\hspace{-1.0cm}\qquad 
= \limsup 
\langle A u_n , u_n - z \rangle
+ \lim \langle A u_n , z - z_n \rangle
= \limsup \langle A u_n , u_n - z_n \rangle\\[2mm]
&&
\quad 
\le \lim \langle - f_n, z_n - u_n \rangle
+ \limsup
j^0 (M u_n, M u_n; M z_n - M u_n) \\[2mm] 
&&\qquad 
+ \limsup \, (\phi_n (z_n) - \phi_n (u_n))\\[2mm]
&&\qquad \quad 
\le \langle - f, z-u^* \rangle
+ j^0 (M u, M u; M z - M u) + \phi (z) - \phi (u^*)
\end{eqnarray*}
for all $z \in U(u^*)$.
In consequence, we deduce that $u^* \in U(u^*)$ is a solution to $P(f, \phi)$.

To conclude the proof, we show the strong convergence 
$u_n \to u^*$ in $V$.
Similarly as before, 
we use condition $(m_2)$ of the Mosco convergence 
for $u^* \in U(u^*)$ to find a sequence 
$\{ \eta_n \} \subset U(u_n)$ such that 
$\eta_n \to u^*$ in $V$, as $n \to \infty$.
We choose $\eta_n$ as a test function in
$P(f_n, \phi_n)$ to get 
$$
\langle A u_n, u_n -\eta_n \rangle 
\le 
\langle - f_n, \eta_n - u_n \rangle
+ j^0 (M u_n, M u_n; M \eta_n - M u_n) 
+\phi_n (\eta_n) - \phi_n (u_n)
$$ 
for all $n \in \nat$. 
Exploiting $H(f)_1$, $H(\phi)_1$(iv) and 
the convergences 
$u_n-\eta_n \rightharpoonup 0$ in $V$, 
$\eta_n\to u^*$ in $V$, we have
\begin{eqnarray}\label{equation42b}
&&\hspace{-0.2cm}
\limsup \langle Au_n, u_n- u^* \rangle
\le \limsup \langle Au_n, u_n - \eta_n \rangle 
+ \limsup \langle Au_n, \eta_n- u^* \rangle
\\ [2mm]
&&\quad 
\le 
\lim \langle - f_n, \eta_n - u_n \rangle
+ \limsup j^0 (M u_n, M u_n; M \eta_n - M u_n) 
\nonumber \\[2mm]
&&\qquad
+ \limsup\, (\phi_n (\eta_n) - \phi_n (u_n))
+ \limsup \langle Au_n, \eta_n- u^* \rangle
\le 0. \nonumber 
\end{eqnarray} 
Finally, by $H(A)$(ii) and (\ref{equation42b}), 
we have 
\begin{eqnarray*}
&&
m_A \, \limsup \| u_n - u^* \|^2 \le
\limsup \langle Au_n - Au^*, u_n - u^* \rangle
\\[2mm]
&&\quad
\le
\limsup \langle A u_n, u_n - u^* \rangle
+
\limsup \langle A u^*, u^* - u_n \rangle 
\le 0
\end{eqnarray*}
which entails the strong convergence of $u_n$ to 
$u^*$ in $V$.
This completes the proof of the theorem.
\end{proof}

Choosing constant sequences $\phi_n = \phi$ 
and $f_n = f$ for all $n \in \nat$ in Theorem~\ref{Convergence}, and
using the arguments of that theorem, 
we obtain the following compactness result.
\begin{Corollary}\label{C1}
Under hypotheses of Theorem~$\ref{Theorem1}$, 
the solution set of Problem~$\ref{QVHI}$ 
is compact in $V$.
\end{Corollary}

We complete this section with the following comments. 

(1) 
Under $H(j)$(i) and (ii), condition $H(j)$(iii) 
means that the generalized gradient operator 
$\partial j \colon X \times X \to 2^X$ 
has a graph closed in $X \times X \times X_w$ topology.
If $j$ is independent of the first argument, condition $H(j)$(iii) is automatically satisfied,  see~\cite[Proposition~3.23(ii)]{MOSbook}.

(2) 
Further, if, in addition to $H(j)$, the potential  
$j(w, \cdot)$ is supposed to be convex, then 
the inequality in Problem~\ref{QVHI} reduces to the following elliptic quasi-variational inequality of the second kind:
find $u \in U(u)$ such that 
\begin{equation*}
\langle A u - f, z - u \rangle
+ \psi(z) - \psi(u) \ge 0
\ \ \mbox{\rm for all} \ \ z \in U(u),
\end{equation*}
where $\psi (z) = j(Mz) + \phi(z)$ for $z \in V$.

(3) 		
In hypothesis $H(j)$ we do not require 
the so-called relaxed monotonicity condition 
of the generalized gradient, extensively used in the 
literature for hemivariational inequalities, see~\cite{MOSbook,MOS30}. In this paper the relaxed 
monotonicity condition is used only in the proof 
of uniqueness of solution.

\section{Well posedness of the Stokes problem}
\label{Section888}

We provide results on existence, uniqueness, and continuous dependence of solution to the inequality in Problem~$\ref{PV}$ on the data. 
To this aim, we apply Theorem~\ref{Theorem1} and Corollary~\ref{C1}.
\begin{Theorem}\label{Theorem1a}
Under the hypotheses $H(T)$, 
$H(f, g)$, $H(h)$, $H(j_\tau)$,  
$H(k, r)$ 	
and the small\-ness condition
\begin{equation}\label{small}
\sqrt{2} \, b_1 \, h_1 \, \| \gamma \|^2 < m_T, 
\end{equation} 
the set of solutions to Problem~$\ref{PV}$ is nonempty 
and compact in $V$.
\end{Theorem}
\begin{proof} 
Let $X = L^2(\Gamma_1;\real^d)$. 
We introduce the following operators and functions
defined by
\begin{eqnarray}
&&\hspace{-1.2cm}
A \colon V \to V^*, \ \ 
\langle A\bu, \bv \rangle = \int_{\Omega}  
{\mathbb{T}}({\mathbb{D}}\bu) : 
{\mathbb{D}} \bv \, dx, 
\ \ \bu, \bv \in V, \label{OPERA} 
\\
&&\hspace{-1.2cm}
J \colon X \times X \to \real, \ \ 
J(\bw, \bu) = \int_{\Gamma_1} h(\bw) j_\tau(\bu)\, d\Gamma, 
\ \ \bw, \bu \in X, \label{OPERJ} \\
&&\hspace{-1.2cm}
\phi \colon V \to \real, \ \ 
\phi(\bv) = 
\int_\Omega g \, \| \mathbb{D} \bv \| \, dx,
\ \ \bv \in V, \label{OPERFI1}\\
&&\hspace{-1.2cm}
\fb_1 \in V^*, 
\ \ 
\langle \fb_1, \bv \rangle = \int_\Omega \fb \cdot \bv \, dx, \ \ \bv \in V, \label{OPERFF} \\[1mm]
&&\hspace{-1.2cm}
M \colon V \to X, \ \ 
M \bv = \bv_\tau, \ \  \bv \in V. \label{OPERM} 
\end{eqnarray}
	
\noindent
We consider the auxiliary variational-hemiva\-ria\-tio\-nal inequality: 
find $\bu \in V$ such that $\bu \in K(\bu)$ and 
\begin{equation}\label{AUX}
\langle A\bu - \fb_1, \bv - \bu \rangle
+ J^0(M\bu, M\bu; M\bv-M\bu) 
+ \phi (\bv) - \phi(\bu) 
\ge 0
\end{equation}
for all $\bv \in K(\bu)$. 
%
We shall verify the hypotheses of Theorem~\ref{Theorem1}.
First we establish hypothesis $H(A)$.
We can use $H(T)$(iii) and the H\"older inequality to obtain
\begin{equation*}
|\int_{\Omega} {\mathbb{T}}({\mathbb{D}}\bu) : 
{\mathbb{D}} \bv \, dx |  
\le 
\sqrt{2} \left(
\| a_0 \|_{L^2(\Omega)}+
a_1 \|\bu \|_{L^2(\Omega;{\mathbb{M}}^d)}
\right) \| \bv \|
\end{equation*} 
for all $\bu$, $\bv \in V$. Hence
$\| A \bu \|_{V^*} \le \sqrt{2} \left(
\| a_0 \|_{L^2(\Omega)} + a_1 \|\bu \| \right)$  
which implies that $A$ is a bounded operator.
From hypothesis $H(T)$(iv), we see that $A$ is a strongly monotone operator 
with constant $m_A=m_T$ as a consequence of 
the inequality 
	\begin{eqnarray*}
		&&\langle A \bv_1 - A \bv_2, \bv_1-\bv_2 \rangle = 
		\int_{\Omega} \left(
		{\mathbb{T}}({\mathbb{D}}\bv_1) 
		- {\mathbb{T}}({\mathbb{D}}\bv_2) \right) : 
		{\mathbb{D}} (\bv_1 - \bv_2) \, dx \\
		&&\qquad 
		\ge m_T \, \int_{\Omega} 
		\| {\mathbb{D}} (\bv_1-\bv_2) \|^2 \, dx = 
		m_T \, \| \bv_1 - \bv_2 \|^2 
		\ \ \mbox{for all} \ \ \bv_1, \bv_2 \in V.
	\end{eqnarray*} 
	Applying~\cite[Theorem~1.5.2]{DMP2} 
	(Krasnoselskii's theorem for the Nemytskii operators)
	together with $H(T)$,   
	we deduce that $A$ is continuous 
	from $V$ to $V^*$. 
	Since the operator $A$ is bounded, monotone and 
	hemicontinuous (being continuous), by~\cite[Theorem~3.69(i)]{MOSbook}, 
	we conclude that $A$ is pseudomonotone, 
	i.e., $H(A)$ holds.
	
We shall check that the function $\phi$ in (\ref{OPERFI1}) satisfies hypothesis $H(\phi)$.
It is obvious that $\phi$ is a convex function.
Hence, it is bounded from below by an affine function 
which combined with the Fatou lemma implies, 
by a standard argument, that it is also lower semicontinuous on $V$ for the strong topology. 
Thus, $H(\phi)$(i) holds.  
Further, by the H\"older inequality and $H(f, g)$, 
we have
\begin{equation*}
	\phi(\bv_1)-\phi (\bv_2) = 
	\int_{\Omega} g \, (\| {\mathbb{D}}\bv_1\|
	- \| {\mathbb{D}}\bv_2\| ) \, dx 
	\le 
	\int_{\Omega} g \, \| {\mathbb{D}} (\bv_1 - \bv_2) \| \, dx 
	\le \| g \|_{L^2(\Omega)} 
	\| \bv_1 - \bv_2 \|
\end{equation*}
for all $\bv_1$, $\bv_2 \in V$, which implies $H(\phi)$(ii).
	
Now, we show that the functional $J$ given by (\ref{OPERJ}) satisfies $H(j)$. 
We use hypotheses $H(j_\tau)$(i)--(iii), $H(h)$, and~\cite[Theorem~5.6.39]{DMP1}, 
to deduce that $J(\bw, \cdot)$ is Lipschitz on every bounded set for all $\bw \in X$, 
which clearly implies condition $H(j)$(i). 
Based on hypothesis $H(j_\tau)$(iv) and~\cite[Theorem~3.47(v), (vii)]{MOSbook}, 
we have
\begin{equation}\label{RRRR}
\partial J(\bw, \bu) = \int_{\Gamma_1}
h(\bw) \partial j_\tau(\bx, \bu(\bx))) \, d\Gamma
\ \ \mbox{for all} \ \ \bw, \bu \in X.
\end{equation}
Let $\bw$, $\bu \in X$ and $\bu^* \in X^*$,  
$\bu^* \in \partial J(\bw, \bu)$.
Hence, 
$\bu^*(\bx) \in h(\bw(\bx)) 
\partial j_\tau(\bx, \bu (\bx))$
for a.e. $\bx \in \Gamma_1$. 
From $H(j_\tau)$(iii) and $H(h)$(iii), we have 
$$
\| \bu^*(\bx)\|^2 \le 
2 \, h_1^2 \, 
(b_0^2(\bx) + b_1^2 \, \| \bu (\bx)\|^2)
\ \ \mbox{for a.e.} \ \ \bx \in \Gamma_1.
$$
Integrating the last inequality on $\Gamma_1$, we obtain
$\| \bu^* \|_{X^*} \le d_1 + d_3 \| \bu \|_X$, 
where 
$d_1 = 2^{1/2} h_1 \|b_0\|_{L^2(\Gamma_1)}$ 
and $d_3 = 2^{1/2}h_1 b_1$.
We infer that the hypothesis $H(j)$(ii) is satisfied with constants $d_1$, $d_2 = 0$, and $d_3$.
	
We shall verify the property $H(j)$(iii). 
Let $\bw$, $\bv$, $\bz \in X$, 
$\bw_n \to \bw$ in $X$, $\bv_n \to \bv$ in $X$, 
and $\bz_n \to \bz$ in $X$. 
From~\cite[Theorem~2.39]{MOSbook}, by passing to a subsequence if necessary, we may suppose  
$$
\bw_n(\bx) \to \bw(\bx), \ \ 
\bv_n(\bx) \to \bv(\bx), \ \
\bz_n(\bx) \to \bz(\bx) \ \ \mbox{in} \ \ \real^d, 
\ \mbox{a.e.} \ \bx \in \Gamma_1
$$
and 
$\| \bw_n (\bx) \|_{\real^d} \le w_0(\bx)$, 
$\| \bv_n(\bx)\|_{\real^d} \le v_0(\bx)$, 
$\| \bz_n (\bx) \|_{\real^d} \le z_0(\bx)$
a.e. on $\Gamma_1$ with $w_0$, $v_0$, $z_0\in L^2(\Gamma_1)$. 
We use the continuity of $h(\bx, \cdot)$ for a.e. $\bx \in \Gamma_1$ and the upper semicontinuity of 
$j_\tau^0(\bx, \cdot; \cdot)$ for 
a.e. $\bx \in \Gamma_1$, see~\cite[Proposition~3.23(ii)]{MOSbook},
to obtain 
\begin{eqnarray*}
&&\hspace{-0.7cm}
\limsup h(\bw_n(\bx))
j_\tau^0(\bx, \bv_n(\bx); \bz_n(\bx))
\le
\limsup h(\bw(\bx))
j_\tau^0(\bx, \bv_n(\bx); \bz_n(\bx)) 
\\[1mm]
&&
+  \limsup (h(\bw_n(\bx))-h(\bw(\bx)))
j_\tau^0(\bx, \bv_n(\bx); \bz_n(\bx))
\le 
h(\bw(\bx))j_\tau^0(\bx, \bv(\bx); \bz(\bx))
\end{eqnarray*} 
for a.e. $\bx \in \Gamma_1$. 
We apply the Fatou lemma,
the regularity hypothesis $H(j_\tau)$(iv) and~\cite[Theorem~3.47(iv)]{MOSbook} to get
\begin{eqnarray}
&&
\limsup J^0(\bw_n, \bv_n; \bz_n) \le
\limsup \int_{\Gamma_1} h(\bw_n(\bx))
j_\tau^0(\bx, \bv_n(\bx); \bz_n(\bx)) \, d\Gamma
\label{EQUALITY}
\\ 
&&\quad 
\le \int_{\Gamma_1}
\limsup h(\bw_n(\bx)) 
j_\tau^0(\bx, \bv_n(\bx); \bz_n(\bx))
\, d\Gamma \nonumber \\ 
&&\qquad 
\le 
\int_{\Gamma_1} h(\bw(\bx))
j_\tau^0(\bx, \bv(\bx); \bz(\bx)) 
\, d\Gamma = J^0(\bw, \bv; \bz), 
\nonumber 
\end{eqnarray} 
This proves $H(j)$(iii) and concludes the proof of condition $H(j)$.

We prove that the map $K \colon V \to 2^V$ 
defined by (\ref{DEFU}) satisfies hypothesis $H(U)$.  
By hypothesis $H(k, r)$, since $k$ is positively homogeneous and $r(\bu) > 0$, we have
$0 = k(\bzero) < r(\bu)$. Hence $\bzero \in K(\bu)$ 
and $K(\bu)$ is nonempty for all $\bu \in V$.
Let $\bu \in V$ and
$\{\bv_n\}\subset K(\bu)$ such that 
$\bv_n\to \bv$ as $n\to \infty$ with $\bv\in V$. 
By the weak lower semicontinuity of $k$, 
we have
$k(\bv) \le \liminf k(\bv_n)\le r(\bu)$.
Thus, the set $K(\bu)$ is closed for all $\bu\in V$.
For any $\bu\in V$, let $\bv_1$, $\bv_2 \in K(\bu)$ 
and $\lambda \in(0,1)$ be arbitrary. The convexity of $k$ (since $k$ is positively homogeneous and subadditive) 
implies
\begin{equation*}
k(\lambda \bv_1 + (1-\lambda) \bv_2)\le 
\lambda k(\bv_1) + (1-\lambda) k(\bv_2) \le 
\lambda r(\bu) + (1-\lambda) r(\bu) = r(\bu),
\end{equation*}
and so $\lambda \bv_1 + (1-\lambda) \bu_2 \in K(\bu)$. 
Hence, $K(\bu)$ is a convex set for all $\bu\in V$.
We deduce that the set-valued map 
$K \colon V \to 2^V$ has nonempty, closed, and convex values.
	
Let $\{ \bu_n \}\subset V$ be such that 
$\bu_n \rightharpoonup \bu$ in $V$ as $n\to \infty$ 
for some $\bu\in V$.
We shall verify that $K(\bu_n) \, \stackrel{M} {\longrightarrow} \, K(\bu)$ by checking 
conditions $(m_1)$ and $(m_2)$ of Definition~\ref{DefMosco}. 
To prove condition $(m_1)$, let 
$\{ \bv_n \}\subset V$ be such that
$\bv_n\in K(\bu_n)$ and 
$\bv_n \rightharpoonup \bv$ in $V$
as $n\to \infty$ for some $\bv \in V$. 
We use the weak continuity of $r$, 
the weak lower semicontinuity of $k$ 
to obtain
$k(\bv)\le \liminf k(\bv_n)
\le \liminf r(\bu_n) = r(\bu)$. 
Thus, $\bv\in K(\bu)$, and implies $(m_1)$. 
For the proof of $(m_2)$,
let $\bv \in K(\bu)$ be arbitrary and  set $\bv_n=\frac{r(\bu_n)}{r(\bu)} \bv$.
Then,
by using the positive homogeneity of $k$, it follows
\begin{equation*}
\displaystyle 
k(\bv_n)=\frac{r(\bu_n)}{r(\bu)} k(\bv)\le r(\bu_n),
\end{equation*}
which implies $\bv_n\in K(\bu_n)$ for every $n\in\mathbb N$.
By the weak continuity of $r$, we have
\begin{equation*}
\lim_{n\to \infty}\|\bv_n-\bv\|
= \lim_{n\to\infty}
\big\|\frac{r(\bu_n)}{r(\bu)}\bv-\bv\big\|
=\lim_{n\to \infty}
\frac{|r(\bu_n)-r(\bu)|}{r(\bu)} \| \bv \| = 0,
\end{equation*}
which entails 
$\bv_n\to \bv$ in $V$ as $n\to \infty$. 
Hence, condition $(m_2)$ follows.
The condition $H(U)$ is verified.
	
From (\ref{tracecompact}), we deduce that
$M$ defined by (\ref{OPERM}) is bounded, linear and compact, and therefore, $H(M)$ holds.
Finally, the smallness condition (\ref{SSMMAALL}) 
of Theorem~\ref{Theorem1} 
is a consequence of (\ref{small}).
Having verified all hypotheses of Theorem~\ref{Theorem1},
we deduce from it that the auxiliary inequality problem  (\ref{AUX}) has a solution. 
From $H(h)$ and $H(j_\tau)$, 
we get the equality 
\begin{equation*}
J^0(\bw, \bv; \bz) = \int_{\Gamma_1} h(\bw(\bx))
j_\tau^0(\bx, \bv(\bx); \bz(\bx)) 
\, d\Gamma \ \ \mbox{for all} \ \ \bw, \bv, \bz \in X. 
\end{equation*}
Using the latter we conclude that $\bu \in V$ is a solution  
to the problem (\ref{AUX}) if and only if $\bu \in V$ 
is a solution to Problem~\ref{PV}.
The compactness of the solution set is a consequence of Corollary~\ref{C1}.
This completes the proof of the theorem.
\end{proof} 

Under more restrictive hypotheses on the data, we obtain 
uniqueness of solution to Problem~$\ref{PV}$.
\begin{Proposition}\label{remkol2} 
\rm
Assume the hypotheses of Theorem~\ref{Theorem1a} and 
	
\noindent 
(a) \ $K$ is independent of $\bu$,
	
\noindent 
(b) \  
the relaxed monotonicity condition holds:
there exists $m_j \ge 0$ such that
$$
j_\tau^0(\bx, \bxi_1; \bxi_2-\bxi_1) + j_\tau^0(\bx,\bxi_2; \bxi_1-\bxi_2) 
\le m_j \, \|\bxi_1 - \bxi_2\|^2
$$
for all $\bxi_1$, $\bxi_2 \in \mathbb{R}^d$ and 
a.e. $\bx \in \Gamma_1$,
	
\noindent 
(c) \  
the condition  
$$h_1 \, m_j \, \| \gamma \|^2 < m_T$$ holds. 
Then Problem~$\ref{PV}$ is uniquely solvable. 
\end{Proposition}
\begin{proof}
Let $\bu_1$, $\bu_2 \in K$ be solutions to 
Problem~\ref{PV}, that is, for $i=1$, $2$, we have
\begin{equation*}
\langle A \bu_i - \fb_1, \bv - \bu_i \rangle
+ J^0(M\bu_i, M\bu_i; M\bv - M\bu_i) 
+ \, \phi (\bv) - \phi (\bu_i) \ge 0
\ \ \mbox{\rm for all} \ \ \bv \in K.
\end{equation*}
Choosing $\bv = \bu_2$ in the inequality for $i=1$ and 
$\bv=\bu_1$ in the inequality for $i=2$, then adding them, we get
\begin{equation*}
\langle A \bu_1-A \bu_2, \bu_1 - \bu_2 \rangle
\le 
J^0 (\bu_{1\tau}, \bu_{1\tau}; \bu_{2\tau}-\bu_{1\tau}) 
+ J^0 (\bu_{2\tau}, \bu_{2\tau}; \bu_{1\tau} - \bu_{2\tau}).
\end{equation*}
	
\noindent 
Exploiting $H(T)$(iv), $H(h)$(iii), the boundedness of 
the operator $M$ and hypothesis (b) in the latter, 
we obtain
$$
(m_T - h_1 m_j \|\gamma\|^2) \| \bu_1 - \bu_2 \|^2 \le 0. 
$$
Finally, by hypothesis (c), we have $\bu_1=\bu_2$ 
which completes the proof.
\end{proof}


Note that the hypothesis (b) in Proposition~\ref{remkol2} 
is equivalent to the condition
$$
(\partial j_\tau(\bx, \bxi_1)-
\partial j_\tau(\bx, \bxi_2)) \cdot (\bxi_1-\bxi_2)
\ge - m_j \, \| \bxi_1 - \bxi_2\|^2
$$
for all $\bxi_1$, $\bxi_2 \in \real^d$, 
a.e. $\bx \in \Gamma_1$, 
known as the relaxed monotonicity condition of 
the subgradient. 
If $j_\tau (\bx, \cdot)$ is a convex function for a.e. $\bx \in \Gamma_1$,
then the latter is satisfied with $m_j = 0$, 
due to the monotonicity of the convex subdifferential.

We conclude this section with a corollary on the dependence 
of the solution set to Problem~$\ref{PV}$ on the data 
$\fb$ and $g$. For simplicity, we suppose that
the plasticity yield stress $g$ is a constant.
We need the following hypotheses. 

\medskip

\noindent 
$\underline{H(f)_2}:$ \quad 
$\fb$, $\fb_n \in L^2(\Omega;\real^d)$, 
$\fb_n \rightharpoonup \fb$ in $L^2(\Omega;\real^d)$.

\medskip

\noindent 
$\underline{H(g)}:$ \quad 
$g$, $g_n \ge 0$, $g_n \to g$.

\begin{Corollary}\label{C123b}
	Assume hypotheses $H(T)$, $H(f)_2$, $H(g)$, 
	$H(j_\tau)$, $H(k, r)$ and $(\ref{small})$. 
	Let $\{ \bu_n \} \subset V$ be a sequence of solutions 
	to Problem~$\ref{PV}$ corresponding to 
	$(\fb_n, g_n)$.
	Then, there is a subsequence of $\{ \bu_n \}$ 
	which converges in $V$ to a solution 
	$\bu \in V$ of Problem~$\ref{PV}$ corresponding to $(\fb, g)$.

\end{Corollary} 
\begin{proof}
Let $\phi$, $\phi_n \colon V \to \real$ be defined 
by (\ref{OPERFI1}) and 
\begin{equation*}
\phi_n(\bv) = g_n
\int_\Omega \| \mathbb{D} \bv \| \, dx 
\ \ \mbox{for} \ \ \bv \in V,
\end{equation*}
respectively.
We shall verify conditions $H(\phi)_1$ and $H(f)_1$. 
We use the compactness of the embedding 
$L^2(\Omega;\real^d) \subset V^*$ to deduce that 
condition $H(f)_1$ follows from $H(f)_2$. 
From $H(g)$ and the estimate
$$
\phi_n(\bv_1)-\phi_n (\bv_2) 
\le g_n \int_{\Omega} \| {\mathbb{D}} (\bv_1 - \bv_2) \| \, dx 
\le  g_n \, \sqrt{|\Omega|}\, \| \bv_1 - \bv_2 \|
$$
for all $\bv_1$, $\bv_2 \in V$, 
it is obvious that $H(\phi)_1$(i)--(iii) are satisfied.
We will check $H(\phi)_1$(iv).
From the weak lower semicontinuity of the norm, 
we obtain
\begin{eqnarray*}
&&
\limsup \, (\phi_n(\bv) - \phi_n(\bu_n)) 
= \limsup g_n \Big( 
\int_\Omega \| \mathbb{D} \bv \| \, dx - 
\int_\Omega \| \mathbb{D} \bu_n \| \, dx
\Big) \\
&&\quad 
\le 
(\limsup |g_n - g|) 
\int_\Omega \| \mathbb{D} \bv \| \, dx
+ \limsup \Big(
(g-g_n) 
\int_\Omega \| \mathbb{D} \bu_n \| \, dx \Big) \\
&&\qquad 
+ \, g \int_\Omega \| \mathbb{D} \bv \| \, dx
- g \, \liminf \int_\Omega \| \mathbb{D} \bu_n \| \, dx 
= \phi(\bv) - \phi(\bu)
\end{eqnarray*}
for all $\bu \in V$, 
$\bu_n \rightharpoonup \bu$ in $V$ 
and all $\bv \in V$. Hence $H(\phi)_1$ is verified.
We apply Theorem~\ref{Convergence},  
and deduce the conclusion of the corollary.
\end{proof}

From Corollary~\ref{C123b} with 
$\mu(r) = \mu_0$ and $\fb_n = \fb$, 
we deduce that for $g_n \to 0$ the Bingham fluid 
tends to behave as a Newtonian one.

\smallskip

To the best of our knowledge, it will be interesting
to explore under what conditions it is possible to recover the pressure from the quasi variational-hemivariational inequality 
in Problem~\ref{PV}.
Also, an interesting topic is to use the continuous dependence results obtain in this paper to study 
optimal control problems and inverse problems.
We are interested in the further extension of the results to time dependent problems.
Finally, it would be challenging to analyze the Bingham model with mixed boundary conditions numerically.

\end{document}